\documentclass{article}[11pt]
\usepackage{epsfig,amsfonts}
\usepackage{amsmath}
\usepackage{amsthm}
\usepackage{graphicx,color}
\usepackage{pgf}
\usepackage{url}
\usepackage{colortbl}
\usepackage{algorithm}
\usepackage{lineno}
\usepackage{mathtools}
\usepackage{adjustbox}
\usepackage{algpseudocode}
\usepackage{caption}
\usepackage{setspace} 

\newtheorem{lemma}{Lemma}[section]

\counterwithin{equation}{section}

\newcommand\Tstrut{\rule{0pt}{2.6ex}}       

\newcommand{\resetcounters}{\setcounter{equation}{0} \setcounter{figure}{0}
 \setcounter{table}{0}}
 
\begin{document}

\title{\bf{Row-Splitting ILU Preconditioners for Sparse Least-Squares Problems}}

\author{Jennifer Scott\thanks{
STFC Rutherford Appleton Laboratory,
Harwell Campus, Didcot, Oxfordshire, OX11 0QX, UK
and School of Mathematical, Physical and Computational Sciences,
University of Reading, Reading RG6 6AQ, UK.
Correspondence to: {\tt jennifer.scott@stfc.ac.uk.}}
\and Miroslav T\r{u}ma\thanks{
Department of Numerical Mathematics, Faculty of Mathematics and Physics,
Charles University, Czech Republic. {\tt mirektuma@karlin.mff.cuni.cz.} 
}
}

\maketitle

\begin{abstract}
Preconditioning for overdetermined least‑squares problems has received comparatively little attention, and designing methods that are both effective and memory‑efficient remains challenging. We propose a class of ILU‑based preconditioners built around a row‑splitting strategy that identifies a well‑conditioned square submatrix 
 via an incomplete LU factorization and combines its incomplete factors with algebraic corrections from the remaining rows. This construction avoids forming the normal equations and is well suited to problems for which the normal matrix
 is ill‑conditioned or relatively dense. Numerical experiments on test problems arising from practical applications illustrate the effectiveness of the proposed approach when used with a Krylov subspace solver and demonstrate it can outperform preconditioners based on incomplete Cholesky factorization of the normal equations, including for sparse–dense problems, where the splitting naturally isolates dense rows.
\end{abstract}

\resetcounters
\section{Introduction}

Linear least-squares (LLS) problems arise in a wide range of practical applications. In this work, we focus on the overdetermined problem
\begin{equation}\label{eq:ls}
\min_{x} \|b - A x\|,
\end{equation}
where $b \in \mathbb{R}^m$, $A \in \mathbb{R}^{m \times n}$ with $m > n$, and $\|\cdot\|$ denotes the Euclidean norm.
If $A$ has full column rank, the problem admits a unique solution $x \in \mathbb{R}^n$.
Our interest lies in large-scale problems for which $A$ is sparse.

LLS problems may be solved by direct, iterative, or hybrid methods.
Direct methods rely on matrix factorizations and are designed to be robust, delivering solutions of predictable accuracy. However, they require explicit access to $A$, and for large-scale problems their computational cost in terms of time and memory can be prohibitive. Moreover, for highly ill-conditioned problems, maintaining numerical stability may be challenging. In many applications, the high accuracy provided by direct solvers may also be unnecessary or unjustified by the data.

These considerations make iterative methods an attractive alternative. Their effectiveness, however, depends critically on the availability of suitable preconditioners to accelerate convergence and limit the number of iterations required to achieve the desired accuracy. This is particularly important in applications where real-time solutions are needed and matrix–vector products with $A$ and $A^T$ are expensive, such as in data assimilation for numerical weather forecasting.

The importance of large-scale LLS problems has motivated extensive research into preconditioning strategies, particularly for
use with Krylov subspace methods CGLS \cite{hest:52}, LSQR \cite{pasa:82}, and LSMR \cite{fosa:2011}.
If $A$ has full column rank, CGLS and LSQR are mathematically equivalent to the conjugate gradient method applied to the symmetric positive definite normal equations
\begin{equation}\label{eq:normal}
Cx = A^T b, \qquad C = A^T A,
\end{equation}
while LSMR is equivalent to MINRES applied to \eqref{eq:normal}.
Relevant literature on preconditioning for LLS problems includes \cite{ardu:15, betu:03a,bjyu:99, cuha:09,jeaj:84,lisa:06,moha:2013,scot:2017}. Comprehensive treatments are provided in the recent monograph by Bj{\"o}rck~\cite{bjor:2024} and the review article \cite{sctu:2025}. A comparative study of direct methods and algebraic preconditioners a decade ago \cite{gosc:2015b, gosc:2017} highlighted the lack of robust general-purpose preconditioners for LLS problems. Although further progress has been made \cite{aljs:2022,ardu:2015,hoba:2018,wath:2022} and because no single approach works well for all problems, the development of effective and broadly applicable preconditioners remains an open challenge.

In this work, we investigate preconditioners based on row splitting of the matrix $A$. A standard row splitting takes the form
\begin{equation}\label{eq:splita}
A = \begin{pmatrix}
    A_1 \\ A_2
\end{pmatrix} ,
\end{equation}
where $A_1$ is a $k \times n$ matrix with $k \ge n$. In practice, the rows of $A$ are permuted to obtain \eqref{eq:splita}, but for simplicity we omit the permutation from the notation. Such approaches are sometimes referred to as subset preconditioners \cite{bjor:2024}.
In previous work, we studied
splittings based on row densities, in which the number $m-k$ of dense rows is small (see
 \cite{sctu:2017b,sctu:2019a,sctu:2021a} and the references therein). Here, we focus on the case $k = n$; $A$ may be sparse or some of the rows may be classified as dense. 
 We propose a class of ILU‑based row‑splitting preconditioners for large LLS problems. An incomplete LU factorization of the rectangular matrix 
$A$ is used to identify a well‑conditioned soarse square block $A_1$
 and to construct an additive‑correction preconditioner that avoids forming the normal matrix 
$C$, which is often significantly denser and more ill‑conditioned than 
$A$. Efficient algorithms are developed for applying the preconditioner, including the treatment of the auxiliary system associated with the remaining rows. To support this framework, we introduce a rectangular variant of an ILUP factorization with threshold pivoting, that is designed to maintain sparsity and ensure robustness. Numerical experiments show that the resulting preconditioners are effective for a range of problems, including quasi‑square and sparse‑dense cases, and offer competitive performance relative to incomplete Cholesky preconditioning of the normal equations.

The remainder of the paper is organised as follows.
Section~\ref{sec:row-splitting} reviews previous work on row-splitting techniques for sparse LLS problems. Section~\ref{sec:splitting-strategies} describes the proposed
splitting-based approach and the practical aspects of the construction of the ILU-based preconditioner are discussed in Section~\ref{sec:ILUP}. Numerical experiments on a range of challenging problems coming from practical applications are presented in Section~\ref{sec: experiments}, including comparisons with an established
incomplete Cholesky factorization preconditioner. Conclusions are drawn in Section~\ref{sec: conclusions}.

\section{Previous approaches involving row splitting}\label{sec:row-splitting}

Row-splitting was introduced in the early 1960s by L{\"a}uchli \cite{lauc:61}. 
He took a splitting  
(\ref{eq:splita}) with $A_1$ square and non-singular and employed CGLS with the preconditioner $M^{-1}=A_1^{-1}$.
With a conformal splitting of the vector $b$, the transformed problem
becomes
\[
\min_y \left\| \begin{pmatrix} b_1  \\ b_2 \end{pmatrix} -
\begin{pmatrix} I  \\ A_2 A_1^{-1} \end{pmatrix} y 
 \right\|,\qquad y = A_1 x.
\]
The application of $ A_1^{-1}$ can be performed by computing its LU factorization, possibly incorporating  pivoting for numerical stability \cite{gipe:88, saun:79}.
Convergence of the resulting algorithm is discussed by Freund~\cite{freu:87}.
Bj{\"o}rck and Yuan~\cite{bjyu:99}
use LSQR and study the conditioning of the 
preconditioned matrix (see also \cite{bjor:2024}). 
In particular, because $A_2A_1^{-1}$ has at most $m-n$ distinct singular values, fast convergence can be expected for quasi-square LLS problems, that is, problems with $m-n$ small,. These arise, for example, in the simulation of multibody systems such as robot arms or satellites \cite{cadj:98}, and in the linear least-squares subproblems occurring in variational data assimilation \cite{nla_da:2025}.

Ideally, the splitting should be such that $A_1$ captures as much relevant information about the problem as possible. In some applications, the structure of $A$
may naturally suggest an appropriate choice (e.g., geodetic applications given in \cite{plem:78}). Elsewhere, the splitting may be motivated by specific properties of the problem itself \cite{pore:69}. For instance, $A$ may have a block of rows that are significantly more weighted or denser than the remainder. 
Another possibility is to find a well-conditioned square submatrix $A_1$ of $A$. The quality of $A_1$ can be measured by the modulus of its determinant (also called the volume of the submatrix). This quantity is maximized in the approach introduced in \cite{tyyz:97}; see also  \cite{goty:01}. Arioli and Duff~\cite{ardu:2015} seek to 
split the rows of $A$ to 
reduce the Euclidean norm of $AA_1^{-1}$; they then factorize $A_1$ and use it to precondition
a symmetric quasi-definite linear system that has the same solution
as the original least-squares problem.

A seminal approach to solving linear least-squares problems that is closely tied to the idea of effective row splitting was introduced by Peters and Wilkinson in 1970 \cite{pewi:70}. This method avoids constructing the 
potentially ill-conditioned normal matrix $C$ by computing a rectangular LU factorization of the (dense) matrix $A$ with pivoting,
that is,
\begin{equation}\label{eq:LUa}
P_r A P_c  =  \begin{pmatrix}A_1 \\ A_2 \end{pmatrix} = L\,U = \begin{pmatrix}L_1 \\ L_2 \end{pmatrix} U.
\end{equation}
Here $L_1$ and $U$ are $n \times n$ lower and upper triangular matrices,
respectively, with $L_1$ having 1's on its diagonal. The row and column permutations $P_r$ and $P_c$ are chosen so that the  off-diagonal entries of $L$ are bounded. This is likely to keep $L=\{l_{ij}\}$ well-conditioned; ill-conditioning in $A$ is reflected in $U$
(see \cite{bjor:76, saun:79} for early work in  the 1970s
and further discussion is given in the monograph \cite{bjor:2024}).
The solution $x$ of the original LLS problem is obtained  by solving 
\begin{equation}\label{eq:PW}
\min_y \| P_r b - Ly \, \|, \qquad UP_c \,x = y.
\end{equation}
The corresponding normal equations $L^T L y = L^TP\,b$
are referred to as the $L$-normal equations.
   This approach has been reported on by Howell and Baboulin~\cite{hoba:2016}; see also the recent discussion of Wathen~\cite{wath:2025}.
   
Bj{\"o}rck and Duff~\cite{bjdu:80} extend the Peters-Wilkinson  method to sparse systems by relaxing the stability
criterion during the LU factorization, choosing the permutations $P_r$ and $P_c$ in (\ref{eq:LUa})
so that $L_{ij} \le \tau$ for some modest $\tau$ , thus seeking to balance numerical stability
with preserving sparsity in the factors. 
However, the threshold-based pivoting strategies 
used by sparse direct linear solvers that offer the greatest guarantees of stability are the 
most computationally intensive and can lead to excessive fill-in in the factors \cite{bjdu:80, hoba:2018}.  
Howell and Baboulin~\cite{hoba:2018} address this by computing an LU factorization with threshold partial pivoting (trading stability for computational speed) and then solve the $L$-normal equations
using preconditioned LSQR with $L_1$ as the preconditioner.

A potential difficulty of Peters-Wilkinson type approaches is that solving (\ref{eq:PW}) necessitates a final back substitution involving $U$.
In some applications, $U$ can have nearly zero diagonal entries. This is seen in \cite{hoba:2018}, where a significant number of the problems 
in the test set taken from the Suite Sparse Matrix Collection \cite{dahu:2011} are discarded because of small entries on the diagonal on $U$. Unfortunately, near-singularity of $U$ 
typically remains undetected until the LU factorization of $A$ has been performed.
Moreover, even with so-called rank-revealing pivoting strategies and employing double precision
arithmetic, in practice it is not straightforward to decide whether a small diagonal entry would be zero in exact arithmetic (and hence whether the problem is rank deficient).

For problems in which $A$ is of full rank and is quasi square, Cardenal et al \cite{cadj:98}
use the LU factors of $A$ to compute a null space of $A^T$. The vector $b$ is
then projected onto this subspace and the least squares solution obtained from the solution of the resulting reduced problem. For some relatively small examples, significant savings compared to the Peters-Wilkinson approach are reported.

We observe that an alternative to using an LU factorization is to employ a QR factorization.
This can offer advantages in terms of stability but for large sparse problems, it can be prohibitively expensive to compute and 
typically results in denser factors. QR factorizations are not used
in this paper.

\section{Row splitting-based preconditioning}\label{sec:splitting-strategies}
We are interested in exploring building
effective preconditioners for large sparse LLS problems by  
choosing $A_1$ using an incomplete factorization and then combining 
the incomplete factors of $A_1$ with information from $A_2$.
The novelty lies in using an incomplete factorization to obtain the row splitting
and then using the incomplete factors and $A_2$ to construct an algebraic preconditioner.
Applying the preconditioner involves a direct solver or using an iterative solver.

Let us assume that we have an approximate solution $x^{(0)}$ of the LLS problem ($x^{(0)}$ could be zero). Our approach uses the idea of approximating an additive correction. To motivate this, consider iterative refinement for solving the normal equations  \cite{mole:67,wilk:63}. This seeks to obtain an additive correction to $x^{(0)}$ by solving the correction equation 
\begin{equation*}
C\,\delta=A^Tr^{(0)},  \quad r^{(0)} =   b - A\, x^{(0)}, \qquad C = A^TA.
\end{equation*}
If solved exactly for the correction $\delta$ then the exact least squares solution 
$x = \delta+ x^{(0)}$ is obtained in a single step.
In practice, the corrected solution is also an approximate solution because the correction equation is solved approximately (for example, using an approximate factorization
of $C$, as discussed in \cite{cada:2025,cahi:2017,cahi:2018}). 
Now consider the factorization-free variant of the preconditioned CGLS method
outlined  in Algorithm~\ref{alg:pcgls1} (see, for example, \cite{hahn:2023}).
This is also called left preconditioned CGLS because it is mathematically equivalent to left preconditioning the normal equations
\[ M^{-1} C \,x = M^{-1} A^Tb.\]
On the first iteration step of Algorithm~\ref{alg:pcgls1}, the direction used to update the initial solution $x^{(0)}$ is based on the corresponding initial residual $r^{(0)}$. On subsequent steps, the search direction is determined by applying a transformation $M^{-1}A^T$ that approximates applying $C^{-1}A^T$ to the latest least squares residual (Steps 11 and 12). 

\begin{algorithm}\caption{Factorization-free (left) preconditioned CGLS \vspace{1mm}}
\label{alg:pcgls1}
\textbf{Input:} Matrix $A \in R^{m \times n}$, 
vector $b \in \mathbb{R}^{m}$, initial solution $x^{(0)}\in \mathbb{R}^{n}$ and preconditioner $M^{-1}\in R^{n \times n}$. 
\\
\textbf{Output:} Computed least-squares solution $x$
\vskip+2mm\hrule\vskip+2mm
\setstretch{1.15}\begin{algorithmic}[1]
\State $r^{(0)} = b -Ax^{(0)}$, $\rho^{(0)} = 1$
   \State    $z^{(0)}  =A^T r^{(0)}$ 
   \State $h^{(0)} = z^{(0)} = M^{-1}z^{(0)}$\Comment{$ M^{-1}$ can be combined with $A^T$ i.e., $h^{(0)} = z^{(0)} = M^{-1} A^T r^{(0)}$} 
  \State $\rho^{(0)} = (z^{(0)}, h^{(0)})$    
\State $p^{(0)}=z^{(0)}$
\For{$i=0,1,2, \dots until$ convergence}
  \State $q^{(i)} = A\,p^{(i)}$
  \State $\alpha^{(i)} = \rho^{(i)}/(q^{(i)}, q^{(i)})$
  \State $ x^{(i+1)} = x^{(i)} + \alpha^{(i)} p^{(i)}$
  \State $r^{(i+1)} = r^{(i)} - \alpha^{(i)} q^{(i)}$
 \State    $z^{(i+1)} = A^T r^{(i+1)}$
 \State    $h^{(i+1)} = M^{-1}z^{(i+1)}$ \Comment{$ M^{-1}$ can be combined with $A^T$ i.e., $h^{(i+1)} = M^{-1}A^T r^{(i+1)}$}
  \State $\rho^{(i+1)} = (z^{(i+1)}, h^{(i+1)})$
  \State  $\beta^{(i+1)} = \rho^{(i+1)}/\rho^{(i)}$
  \State $p^{(i+1)} = h^{(i+1)}+ \beta^{(i+1)} p^{(i)}$
\EndFor
\end{algorithmic}
\end{algorithm}

To derive a left preconditioner based on a given row splitting of $A$ and 
the employment of an additive correction, 
we derive - using a closed-form expression that transforms the problem residual - the quantity that must be added to $x^{(0)}$ to obtain the least-squares solution. 
Approximating the additive correction leads to a form of preconditioning (see \cite{sctu:2017b} in the context of sparse-dense least squares problems).
The additive correction is formulated in Lemma~\ref{th:solut}.
Throughout the discussion, we  denote $C_i = A_i^TA_i$ ($i = 1,2$).

\begin{lemma} \label{th:solut} 
Assume a row splitting of the $m \times n$ ($m \ge n$) full rank matrix $A$ of the form
(\ref{eq:splita}),
where the $k \times n$ ($k \ge n$) matrix $A_1$ is of  full rank.
Let $x^{(0)}$ be an approximate solution of the LLS problem (\ref{eq:ls})
 and let $r_1= b_1-A_1x^{(0)}$ and $r_2=b_2-A_2 x^{(0)}$ be the components of its residual vector $r^{(0)}$. 
Then the least-squares solution  of (\ref{eq:ls}) is $x = x^{(0)} + \delta$, where the additive correction $\delta =C^{-1}A^T r^{(0)}$ is given by

\begin{equation}\label{eq:fullspl_1}
  \delta  =  C_1^{-1}(I- A_2^T S^{-1}A_2C_1^{-1})z,
\end{equation}
with
\begin{equation}\label{eq:matrixS}
z= A^T r^{(0)} \quad \mbox{and} \quad S = I+A_2C_1^{-1}A_2^T = I +(A_2A_1^{-1})(A_2A_1^{-1})^T  \in \mathbb
R^{(m-k) \times (m-k)}. 
\end{equation}

\end{lemma}
\begin{proof}
The matrix $C_1$ is non singular.  $S$ is also non singular because it is the sum of a symmetric positive definite matrix and a positive semi definite matrix. The result is derived from a straightforward application of the  Woodbury formula \cite{wood:49, wood:50},
which enables $C^{-1}$  to be written in the form
 \begin{eqnarray*}
C^{-1}&=& (C_1 + C_2)^{-1} \\ &=& (C_1 + A_2^T A_2)^{-1} \\
   &=&  C_1^{-1} - C_1^{-1}A_2^T(I + A_2C_1^{-1}A_2)^{-1}A_2 C_1^{-1}.
\end{eqnarray*}
\end{proof} 
Setting $z = A_1^T r_1 + A_2^T r_2$ and using $S^{-1} = I - S^{-1}A_2C_2^{-1}A_2^T$, we can rewrite
(\ref{eq:fullspl_1})  as
\begin{equation}\label{eq:fullspl}
  \delta =  C_1^{-1}(A_1^Tr_1 -  A_2^T S^{-1} (A_2C_1^{-1}A_1^Tr_1 - r_2)).
\end{equation}
This shows how $\delta$ can be computed using $C_1$ without constructing $C$.
But obtaining (\ref{eq:fullspl}) 
directly is generally impractical. In particular, constructing $S$ and factorizing it may  be computationally too expensive and require too much memory. 
However, by replacing the quantities 
that need to be inverted by appropriate approximations, it may be feasible to employ the resulting
$\widetilde S$ within a preconditioned iterative solver. 
In particular, suppose that we have non singular approximations $M_a^{-1} \approx C_1^{-1}$ and $M_b^{-1} \approx S^{-1}$.
The approximate additive correction $\delta \approx M^{-1} A^T r = M^{-1} z $ can be used to define a preconditioner within Algorithm~\ref{alg:pcgls1}.
The calculation of $h = M^{-1}z$ (in Steps 11 and 12) can be implemented using Algorithm~\ref{alg:apply1}. 
It requires  a matrix-vector product with $A_2$ and with $A_2^T$, plus
 two linear system solves with $M_a$ and one with $M_b$.
 If $m-k$ is small, the operations involving $A_2$ and $M_b$ are cheap.
\begin{algorithm}[htpb]
\caption{Application of the additive correction preconditioner \vspace{1mm}}\label{alg:apply1}
\textbf{Input:} $z \in \mathbb R^{n}$, $A_2  \in \mathbb R^{(m-k) \times n}$, $M_a \approx C_1 \in \mathbb R^{n \times n}$, $M_b \approx S \in \mathbb R^{(m-k) \times (m-k)}$.\\
\textbf{Output:} Preconditioned vector $h =M^{-1}z$.
\vskip+2mm\hrule\vskip+2mm
\setstretch{1.15}\begin{algorithmic}[1]
\State Solve $M_a u_a = z$
\State Compute $w_b =   A_2 u_a$
\State Solve $M_b v_b =w_b$
\State Compute $w_a = u_a - A_2^T  v_b$
\State Solve $M_a h = w_a$
\end{algorithmic}
\end{algorithm}

The general row splitting-based framework offers a number of possibilities for approximating the matrices $M_a$ and $M_b$. 
In the rest of this paper, we concentrate on
the special case $k=n$ 
and $A_1$ is a square non singular submatrix of $A$. We use an LU factorization-based approach to compute and apply the preconditioner
without constructing and factorizing the matrix $C_1$. 

\subsection{$k=n$ and $A$ quasi square: direct solver}\label{subsection:quasi}
We start by assuming we have performed an LU factorization (\refeq{eq:LUa}) of $A$ and that $A$ is quasi square.
Omitting the permutation matrices to simplify the notation, the matrix $S$ from (\ref{eq:matrixS}) becomes
\begin{equation}\label{eq:S}
S = I +(A_2A_1^{-1})(A_2A_1^{-1})^T= I + YY^T, 
\end{equation}
where 
\begin{equation}\label{eq:Y}
Y = A_2A_1^{-1} = L_2 L_1^{-1} \in \mathbb R^{(m-n) \times n}.
\end{equation}
Setting $x^{(0)} = 0$, we have
$r_1 = b_1$, $r_2 = b_2$, and (\refeq{eq:fullspl_1}) becomes
\begin{eqnarray*}\label{eq:square A1}
    x = \delta &=& A_1^{-1}(A_1^{-T} - A_1^{-T} A_2^{-T}S^{-1}A_2A_1^{-1}A_1^{-T})z  \\
    &=& A_1^{-1}(A_1^{-T} -Y^{T} S^{-1}YA_1^{-T})z \\
    &=& A_1^{-1}(b_1 + Y^Tb_2 -Y^{T} S^{-1}Y (b_1+Y^Tb_2)) \\
    &=& A_1^{-1}(b_1 - Y^T S^{-1}Yb_1 +(Y^T - Y^TS^{-1}YY^T)b_2) \\
   &=& A_1^{-1}(b_1 - Y^T S^{-1}Yb_1 + Y^TS^{-1}(S - YY^T)b_2) \\
    &=& A_1^{-1}(b_1 + Y^T S^{-1}(b_2 - Yb_1)).
\end{eqnarray*}
From (\refeq{eq:square A1}), the least squares solution can be computed
by first solving the triangular systems $L_1^T Y^T = L_2^T$, then
forming $S$ and solving the SPD system
\begin{equation}\label{eq:quasi_square_w} 
Sw = b_2 - Yb_1.
\end{equation}
Provided $A$ is quasi square, $S$ 
is small and can be held as a dense matrix and LAPACK routines employed for this solve.
Finally, the factors of $A_1$ can be used to solve the sparse system
\begin{equation}\label{eq:quasi_square_exact}
A_1 x =L_1U x = b_1 + Y^T w.
\end{equation}
We observe that this formulation, which avoids the normal matrices, is also derived in \cite{cadj:98} using the augmented system formulation of the least-squares problem.

\subsection{$k=n$: ILU-based preconditioner}
For large-scale problems we may not want, or be able, to compute an LU factorization of $A$. We then need to find a preconditioner for use with an iterative solver.
We modify the above approach
to replace the LU factorization by an
incomplete (ILU) factorization.
Let the ILU factorization of $A$ be given by 
\begin{equation}\label{eq:ILU}
P_rA P_c   \approx \begin{pmatrix}\widetilde L_1 \\ \widetilde L_2 \end{pmatrix} \widetilde U.
\end{equation}
Here and elsewhere the notation $\,\widetilde{}\,$  is used
to denote incomplete quantities.
Omitting the permutation matrices,  
\begin{equation} \label{eq: S}
\widetilde S = I + \widetilde Y \widetilde Y^T,
\end{equation}
with $\widetilde Y = \widetilde L_2 \widetilde L_1^{-1}$. The exact correction derived in Subsection~\ref{subsection:quasi} and given by (\ref{eq:quasi_square_w}) and (\ref{eq:quasi_square_exact}) can be replaced by the approximate quantities 
$$\widetilde Sw = r_2 - \widetilde Yr_1 \quad \mbox{and} \quad A_1 h = r_1 + \widetilde Y^T w. $$
It follows that, using the residual components $r_1$ and $r_2$, Steps 3 and 12 of Algorithm~\ref{alg:pcgls1} can 
combine the preconditioner with $A^T$.
The application of the preconditioner is given in Algorithm~\ref{alg:apply_ilu}. This requires a matrix-vector product with 
$\widetilde Y$ and $\widetilde Y^T$, and a solve with each of $\widetilde S$, $\widetilde L_1$ and $\widetilde U$.
Importantly, through $\widetilde Y$, the preconditioner involves both $\widetilde L_1$ and $\widetilde L_2$.

\begin{algorithm}[htpb]
\caption{Application of the LU (or ILU) left preconditioner \vspace{1mm}}\label{alg:apply_ilu}
\textbf{Input:} Residual components $r_1, r_2$,  (incomplete) factors $\widetilde L_1$ and $\widetilde U$ of $A_1 \in \mathbb R^{n \times n} $.\\
\textbf{Output:} Preconditioned vector $h =M^{-1}z$, where 
$z = A_1^T r_1 + A_2^T r_2 \in \mathbb R^n$.
\vskip+2mm\hrule\vskip+2mm
\setstretch{1.15}\begin{algorithmic}[1]
\State Compute $u = r_2 - \widetilde Y r_1$ \Comment{$\widetilde Y = \widetilde L_2 \widetilde L_1^{-1}$}
\State Solve $\widetilde S w =u$ \Comment{$ \widetilde S = I + \widetilde Y \,\widetilde Y^T$}
\State Compute $y= r_1 + \widetilde Y^T w$
\State Solve $\widetilde L_1 v =y$ then $ \widetilde Uh = v$
\end{algorithmic}
\end{algorithm}

The matrix $\widetilde Y$ can be computed and stored as a  sparse matrix. 
This requires solving the sparse triangular systems $\widetilde L_1^T \widetilde Y^T = \widetilde L_2^T$.
To do this efficiently, a back substitution routine with $\widetilde L_1^T$
that handles sparse right-hand side vectors corresponding to the $m-n$ columns of $\widetilde L_2^T$ is needed. Because the memory requirements are not known a priori, this is first  determined by performing symbolic substitutions.
Alternatively,  $q = \widetilde Y r_1$ may be computed by solving
$\widetilde L_1 q_1 = r_1$ and then forming the matrix-vector product $q = \widetilde L_2\, q_1$ (and similarly, for $q = \widetilde Y^T w$). Avoiding
computing and storing $\widetilde Y$ explicitly  saves memory and, if the number of applications of the preconditioner is small, it also reduces the total flop count.

The $(m-n) \times (m-n)$ matrix $\widetilde S= I + \widetilde Y \,\widetilde Y^T$ can be assembled row-by-row and factorized as a sparse matrix. Alternatively, it  can be held as a dense matrix and factorized using LAPACK routines. The factorization only needs to be performed once and the factors reused for each application of the preconditioner.
Factorizing $\widetilde S$ is  feasible if $A$ is quasi square. If not, the storage for
$\widetilde S$ dominates that needed for the ILU factors and
limits the size of problems that can be solved. 

For more general $m$ and $n$, $\widetilde S$ may be computed
implicitly as $I + \widetilde L_2 \widetilde L_1^{-1} \widetilde L_1^{-T} \widetilde L_2^T$ (avoiding forming $\widetilde Y$). The linear system $\widetilde S w =u$ in Step 2 
of Algorithm~\ref{alg:apply_ilu} is then solved using an iterative solver. Because $\widetilde S$ is SPD, the conjugate gradient method is the obvious choice. Each matrix-vector product with $\widetilde S$ involves
a matrix-vector product with each of $\widetilde L_2$ and $\widetilde L_2^T$ and a solve with each of $\widetilde L_1$ and $\widetilde L_1^T$.

The use of an iterative solver at Step 2 suggests that $\widetilde S$ 
may be replaced by an approximation $\widehat S \approx I + \widetilde Y \,\widetilde Y^T$ such that $\widehat S w = u$ is less expensive to solve. The simplest choice is $\widehat S = I$, avoiding altogether the need to solve a linear system.

\subsection{Sparse-dense least-squares problems}
\label{sec:sparse-dense}

Least-squares problems that are sparse with the exception of a small number of rows that contain significantly more entries than the other rows can arise in practice (examples and references are given in \cite{sctu:2017b}).
An obvious row splitting is 
$$ P_rA = \begin{pmatrix}
    A_s \\ A_d
\end{pmatrix}, $$
where $A_s$ comprises the $m_s=k$ sparse rows and
$A_d$ is the $m_d= m-k$ rows that are classified as dense.
The matrix $C_s = A_s^T A_s$ is  sparse and symmetric, but not necessarily non singular.
Suppose we have a non singular approximation $M_s \approx C_s$  in factorized form, that is, $M_s = \widetilde L_s \widetilde L_s^T$, where $\widetilde L_s$ is an incomplete Cholesky factor of $C_s(\alpha) = C_s + \alpha I$, where $\alpha >0$ is small and chosen so that $C_s(\alpha)$ is SPD and the factorization does not breakdown.  The $m_d$ triangular systems
$\widetilde L_s \widetilde Y_d^T = A_d^T$ are solved for $ \widetilde Y_d^T$ and then 
the small dense SPD matrix $\widetilde S_d = I+\widetilde Y_d\, \widetilde Y_d^T$ is computed and factorized using dense linear algebra kernels to obtain $M_d \approx \widetilde S_d$.

This approach can be applied for general problems with $k \ge n$ provided $m-k$ is small.
However, it is necessary to predetermine the rows in $A_d$; see \cite{sctu:2021a} for how this might be done. Furthermore, it involves factorizing the normal matrix $C_s$.
By employing  an ILU factorization algorithm that incorporates row
interchanges to preserve sparsity (see Section~\ref{sec:ILUP} below)
both these issues are avoided. Provided $m_d \le m-n$,
the $\widetilde L_1$ factor of $A$ is sparse and at least $m_d$ rows of $\widetilde L_2$ are dense. In this case, using the notation of the previous section,
$\widetilde Y$  contains dense rows and $\widetilde S$ is dense and
Algorithm~\ref{alg:apply_ilu} can again be employed to apply the preconditioner.

\subsection{Adding and removing data}

In some situations, it is necessary to solve a sequence 
of least-squares problems in which a small number of rows are added to or deleted from the least-squares matrix \cite{mmgh:2017}.
Consider first adding a row $a$ to $A$ after its LU (or ILU) factorization (\ref{eq:LUa}) has been computed. The standard procedure for updating the
factorization keeps $L_1$ unchanged and adds an extra row $l$ to $L_2$
that is obtained by solving $(UP_c)^T l = P_ra$. The matrix $Y$
given by  (\ref{eq:Y}) has an additional row and $Y^T$ has an additional column. If $Y^T$ 
is computed explicitly, an extra triangular solve is needed to compute the additional column  (but the other columns are unchanged). The matrix $S$ 
given by (\ref{eq:S}) can be updated to include the contribution from the extra column in $Y^T$. Thus, updating the preconditioner is straightforward and inexpensive provided the factorization remains stable. If adding a  row makes $A$ rank deficient, then the update fails and the new row would require a new pivot for stability, making a refactorization necessary.

For the case that one or more rows are removed from $A$, assume we already have
a preconditioner for $A$,  which we denote by $M_A^{-1} \approx C^{-1}$. Let $A_2$ be the rows of $A$ that are to be removed and set $C_1 = C- C_2$. Provided $A_1$ is of full rank, applying the Woodbury formula, we obtain
$$C_1^{-1} = C^{-1}+ C^{-1}A_2^T(I - A_2C^{-1}A_2^T)^{-1}A_2C^{-1} .$$
$M_A$ can be used in approximating $(I - A_2C^{-1}A_2^T)^{-1}$.
In particular, if $A \approx \widetilde L \widetilde U$ then 
the approach of Algorithm~\ref{alg:apply_ilu} can be adapted to apply the preconditioner for the reduced problem.

\section{ILUP preconditioner construction}
\label{sec:ILUP}
Over the past several decades, a variety of ILU techniques have been proposed for constructing
preconditioners for use with Krylov subspace methods
for solving sparse linear systems; a few have been developed into general-purpose mathematical software for square systems, such as \cite{ilupack, lish:2011}. Key approaches include the level-of-fill structure-based approach, numerical threshold-based methods, and numerical inverse-based multilevel techniques. There are numerous publications in the literature; relevant references include \cite{bosa:06,hypo:01,lish:2011,saad:94,saad:03}.
Here we employ a variant of ILUP, a sophisticated threshold-based ILU scheme 
that uses a dual numerical dropping strategy and incorporates numerical pivoting (see \cite[Chapter 10]{saad:03}).

Threshold-based ILU methods have been successfully applied as black-box approaches for solving a wide range of practical problems. They involve determining which entries to retain in the sparse triangular factors by examining their absolute values and discarding those below a prescribed tolerance~$\tau$. In ILUP($p,\tau$), at each stage of the factorization, entries of absolute value less than $\tau$ are removed from the current column of $L$ and $U$, after which at most $p$ of the largest remaining entries are retained; this controls memory growth.  This approach allows for both structural considerations and numerical information, and 
the inclusion of pivoting seeks to ensure stability while retaining sparsity.

For our numerical experiments, we have developed a prototype Fortran implementation of ILUP($p,\tau$) for rectangular $A \in \mathbb{R}^{m \times n}$
with $m \ge n$; it is based on the sparse (complete) LU factorization algorithm of Gilbert and Peierls~\cite{gipe:88}.
Their algorithm was proposed for square systems. It performs Gaussian elimination
in column-wise order, and splits the computation of each column into
a symbolic phase followed by the
numerical stage that involves solving a sparse triangular
system. Partial pivoting by rows is performed in time proportional
to the arithmetic operations.
We modified the approach to compute an ILU
factorization of rectangular $A$. We incorporate threshold partial pivoting and 
use modifications to avoid small diagonal entries in $U$, allowing us to handle matrices that either loose the linear dependence of the columns because of the incompleteness of the factorization or were originally rank-deficient.

Our ILUP$(p,\tau)$ procedure is summarized in Algorithm~\ref{alg:ILU}. Although the method allows for distinct limits $p_L$ and $p_U$ on the maximum number of retained entries in each column of $\widetilde L$ and $\widetilde U$, introducing these additional parameters offers no consistent performance advantage in our experiments. To avoid unnecessary user tuning, we set $p = p_L = p_U$ throughout.

In addition to the dropping parameters $p$ and $\tau$, the user must specify the pivoting threshold $\mu$ and the parameter \texttt{small}, which bounds the minimum admissible magnitude of the diagonal entries of $\widetilde{U}$. At each step $j$, the pivot is selected using Algorithm~\ref{alg:find_pivot}. The threshold $\mu$ governs the trade-off between numerical stability and sparsity: values close to $1$ emphasize stability, whereas smaller values promote sparsity. Following common practice in sparse LU factorization algorithms, we set $\mu = 0.1$ in all experiments.

Step~6 of Algorithm~\ref{alg:ILU} requires a sparse triangular solve. Although the details are omitted here, we note that a depth-first-search procedure is first employed to determine the sparsity pattern of the vector $\widetilde{U}_{1:j-1,j}$ prior to performing the solve.

\begin{algorithm}\caption{\bf Column-wise ILUP($p, \tau$) factorization with threshold partial  pivoting}\label{alg:ILU}
\textbf{Input:}  Matrix $A\in \mathbb{R}^{m \times n}$ with  $m \ge n$,
dropping controls  $p$ and $\tau$,  $small>0$  controls the diagonal entries of $\widetilde U$, and
threshold pivoting parameter $0 < \mu \le 1$. 

\smallskip
\textbf{Output:} ILU factorization  $PA \approx  \widetilde L \widetilde U$, where $\widetilde L \in \mathbb{R}^{m \times n}$ is lower trapezoidal with 1's on the diagonal, $\widetilde U \in \mathbb{R}^{n \times n}$  is upper triangular and $P \in \mathbb{R}^{m \times m}$ is a row permutation. 
\smallskip\hrule\smallskip

\setstretch{1.17}\begin{algorithmic}[1]
\State Compute row counts $rc(1:n)$ of $A$
\For{$j=1:n$}
  \If {$j=1$}
  \State Set $\widetilde L_{1:m,1} = A_{1:m,1}$
  \Else
  \State  {Triangular solve  $\widetilde L_{1:j-1\,,1:j-1} \widetilde U_{1:j-1,\,j} = A_{1:j-1,\,j}$} 
  \State Compute {$\widetilde L_{j:m,\,j} = A_{j:n,\,j} - \widetilde L_{j:m,\,1:j-1} \widetilde U_{1:j-1,j}$} 
      \State Keep at most $p$ largest entries in $\widetilde U_{1:j-1,\,j}$ of absolute value at least $\tau$
  \EndIf
  \If {$\widetilde L_{j:m,j}$ has no non zero entries}
    \State Add  $\widetilde L_{jj}$ and give it a non zero value such that 
    $|\widetilde L_{jj}| \ge small$ \Comment{Algorithm \ref{alg:modify}} 
  \Else
  \State Choose the pivot row $i$ in column $\widetilde L_{j:m,j}$ \Comment{Algorithm \ref{alg:find_pivot}}
  \State If $|\widetilde L_{ij}| < small$, modify its value \Comment{Pivot too small. Modify it using Algorithm \ref{alg:modify}}
  \State Interchange rows $i$ and $j$ \Comment{Row permutation of $A$}
      \EndIf \Comment{$j$ is now the pivot row}
    \State Scale $\widetilde L_{j+1:m,\,j}=\widetilde L_{j+1:m,\,j}/\widetilde L_{jj}$
      \State Set $\widetilde U_{jj}= \widetilde L_{jj}$ and $\widetilde L_{jj}= 1$ \Comment{$\widetilde U_{jj} \ge small$}
  \State Keep at most $p$ largest entries  in $\widetilde L_{j+1:m,\,j}$ with absolute value at least $\tau$
  \State Update $rc(j+1:m)$ to hold row counts in $A_{j+1:m,\,j+1:n}$.
\EndFor
\end{algorithmic}
\end{algorithm}

\medskip
\begin{algorithm}\caption{\bf Choose  the pivot in column $\widetilde L_{j:m,j}$}\label{alg:find_pivot}
\textbf{Input:} Column $\widetilde L_{j:m,j}$ of the lower trapezoidal matrix $\widetilde L\in \mathbb{R}^{m \times n}, m \ge n$, threshold pivoting parameter $0 < \mu \le 1$,
current row counts $rc(j:m)$ \\
\smallskip
\textbf{Output:} Pivot row $k$  
\smallskip\hrule\smallskip
\setstretch{1.45}\begin{algorithmic}[1]
\State Construct the set $R(j) = \{ \,q \ : \ |\widetilde L_{qj}| \ge \mu \max_i \{ \, |\widetilde L_{ij}| \; : \,j \le q \le m\} \,\}$ 
\State Choose $k \in R(j)$ such that $rc(k) \le rc(q) \ for \ all  \ q \in R(j)$ \Comment{Ties are broken arbitrarily}
\end{algorithmic}
\end{algorithm}

\medskip
\begin{algorithm}\caption{\bf Modify $L_{ij}$ (see \cite{lish:2011})}\label{alg:modify}
\textbf{Input:} Entry $\widetilde  L_{ij}$ to be  modified,  column dimension $n$ of $\widetilde L$, parameter $small > 0$.  \\
\smallskip
\textbf{Output:} Modified  $\widetilde  L_{ij} \ge small$.
\smallskip\hrule\smallskip
\setstretch{1.45}\begin{algorithmic}[1]
\State  Compute $\beta=10^{-2(1-j/n)}$ \Comment{$\beta$ increases with $j$}
\State Set  $\widetilde L_{ij}=\max(\,\beta\, \|\ A_{1:m,j}\|_{\infty}, \;small\,)$       
\end{algorithmic}
\end{algorithm}
During the ILU factorization, even if the matrix $A$ is of full rank and not ill-conditioned, dropping entries can result in a null column and hence a zero pivot. Such breakdown is most likely towards the end of the computation. When a zero (or unacceptably small) pivot occurs, 
we employ the local modification strategy proposed by Li and Shao~\cite{lish:2011}, given here as
Algorithm~\ref{alg:modify}. Modifying an entry $\widetilde L_{ij}$ in this way is based on two observations. First, when there is a (close to) zero pivot it needs to be replaced by a sufficiently large non zero value
so that  within each application of the ILU preconditioner, the solve involving  $\widetilde U$ is stable. 
The second observation is that as the factorization progresses, $\beta=10^{-2(1-j/n)}$ (and hence the modified values) increases  with the column index $j$.
Adding a small perturbation to a small pivot is a simple mechanism  that enables the
incomplete factorization to finish without breakdown and prevents the diagonal entries of $\widetilde U$ from being too small. 

\section{Numerical experiments}
Having introduced our LIUP preconditioner and outlined how it is applied, in this section we report on numerical experiments and illustrate the effectiveness of the proposed ILU preconditioner for a range of problems that originate in practical applications.

\label{sec: experiments}
\subsection{Test environment}

Our test examples are listed in Table~\ref{T:test problems} in increasing order of the ratio $(m-n)/m$. 
We include both quasi square problems (for which the ratio is small) and problems for which $m$ is significantly larger than $n$. The problems
come from either the  SuiteSparse Matrix Collection\footnote{\url{https://sparse.tamu.edu/}}
or the CUTEst linear
programme set\footnote{\url{https://github.com/ralna/CUTEst}} and comprise a subset of those used by Gould and Scott in their study of numerical methods for
solving large-scale least-squares problems \cite{gosc:2015b,gosc:2017}. 
The  examples  that are not SuiteSparse matrices have an identifier of the form Gould\_Scott/problem\_name. Note that some of the 
SuiteSparse matrices have been modified.
Specifically, any null rows or columns are removed
and the matrix transposed to give an overdetermined system ($m > n$).
In all our experiments, $A$ is prescaled so that the 2-norm of each column of the scaled matrix is equal to 1. The problems were chosen because they are classified 
by Gould and Scott as challenging to solve, particularly for iterative
solvers; many are highly ill conditioned and/or have one or more rows that have significantly more entries than the rest (that is, they are
sparse-dense problems).

We take the entries of the vector $b$ to be random numbers in the interval $[-1,1]$.
This results in a non zero residual.
For a given problem, the same $b$ is used for each experiment.

A key practical issue in the implementation of any iterative method is the
choice of stopping critera. It is desirable to consider  criteria based
on the error in the least-squares solution. One possibility is
\begin{equation}
\label{eq:tight_stop}
\frac{\|x_{true} - x^{(i)}\|_{A^T A}}{\|A\| \, \|x^{(i)}\| + \|b\|}
= 
\frac{\|A A^\dagger r^{(i)}\|}{\|A\| \, \|x^{(i)}\| + \|b\|}
\le \delta,
\end{equation}
where $\delta >0$ is the desired accuracy. Here,
$A^\dagger = (A^T A)^{-1} A^T$ is the pseudo-inverse, $x_{true} = A^\dagger \,b$ is exact the minimum-norm least-squares solution and
$\|v\|_{A^T A}^2 = v^T A^T A v$.
This  measures the error in the energy norm induced by $A^T A$
and is directly related to the distance to the exact solution.
Importantly, \eqref{eq:tight_stop} remains valid under scaling
and preconditioning, because it is expressed in terms of the original problem.
It is shown in \cite{cptp:2009} that (\ref{eq:tight_stop}) is asymptotically tight in the limit as $x^{(i)}$ approaches $x_{true}$ and if $\delta = O(u)$ then the computed solution is a backward stable least-squares solution \cite{cptp:2009,jitp:2010}. 

Recent work of Pape\v{z} and Tich\'y~\cite{pati:2024} demonstrates that a practical and inexpensive 
estimate  $estim_{i} \approx \|x_{true} - x^{(i)}\|_{A^T A}^2$ can be computed from quantities
available within the CGLS (or LSQR) iteration. The estimate can be used until the level of
maximal attainable accuracy is reached. We employ the estimator in
\cite{sctu:2025b} and show that is reliable in practical applications.
We therefore terminate CGLS when 
\begin{equation}\label{eq:ratio_PT}
ratio_{PT} = \frac{estim_{i}}{\|A\| \, \|x^{(i)}\| + \|b\|} \le \delta
\end{equation}
An estimate of the 2-norm of $A$ is computed using the power method. In our experiments, we set $\delta = 10^{-10}$.
and limit the number CGLS iterations to 2000.

\begin{table}[htbp]
\caption{Test matrices in order of increasing $(m-n)/m$. $nz(A)$ denotes the number of entries in $A$ and $lmax$ is the largest number of entries in a row.
}
\label{T:test problems}\vspace{3mm}
{
\footnotesize
\begin{center}
\tabcolsep=1.8mm
\begin{tabular}{l|rrrrr} \hline
Problem number and identifier     &   $m$  & $n$ & $nz(A)$ & $(m-n)/m$   & $lmax $ 
 \\
\hline\Tstrut
1. {\tt Meszaros/ex3sta1}       & 17,516 & 17,443 & 68,779 & 0.004 & 46  \\
2. {\tt HB/beaflw}              &    500 &    492 & 53,403 & 0.02  &400  \\
3. {\tt JGD\_Groebner/f855\_mat9} &  2,511 & 2,456 & 171,214 & 0.02 & 829 \\
4. {\tt LPnetlib/lpi\_klein3 }    & 1,082  &   994 & 13,101  & 0.08  & 275 \\
5. {\tt JGD\_Forest/TF14}         & 3,159  & 2,644  & 29,862 & 0.16  & 13  \\
\hline\Tstrut
6. {\tt JGD\_G5/IG5-13}           & 3,994  & 2,532  & 91,209 & 0.16  & 119  \\
7. {\tt JGD\_Forest/TF16}         & 19,320 & 15,437 & 216,173 & 0.20  & 15 \\
8. {\tt JGD\_Groebner/c8\_mat11}  & 5,761  & 4,562  & 2,462,970 & 0.21 & 2,417 \\
9. {\tt Meszaros/scagr7-2c}       & 3,479  & 2,447  & 9,005 & 0.30  & 449 \\
10. {\tt Ycheng/psse1 }            & 14,318 & 11,028 & 57,376 & 0.30 & 18  \\
\hline\Tstrut

11. {\tt Meszaros/scagr7-2b}       & 13,847 &  9,743 & 35,885 & 0.30 & 1,793  \\
12. {\tt LPnetlib/lp\_ken\_11}     & 21,349 & 14,694 & 49,058 & 0.31 & 3  \\
13. {\tt Gould\_Scott/GE}          & 16,369 & 10,099 & 44,825 & 0.38 & 36  \\
14. {\tt Gould\_Scott/LARGE001} & 7,176 &  4,162 & 18,887 & 0.42 & 11  \\
15. {\tt Gould\_Scott/DELF000}     &  5,543 &  3,128 & 13,741 & 0.44 & 9  \\
\hline\Tstrut

16. {\tt Meszaros/scrs8-2r}        & 27,691 & 14,364 & 58,439 & 0.48 & 2,052 \\
17. {\tt Ycheng/psse0 }            & 26,722 & 11,028 & 102,432 & 0.59 &  4 \\
18. {\tt Gould\_Scott/PILOTNOV}    &  2,446 &    975 &  13,331 & 0.60 &  40\\
19. {\tt HB/illc1850 }             &  1,850 &    712 &   8,758 & 0.62 &  5 \\
20. {\tt Ycheng/psse2 }            & 28,634 & 11,028 &  115,262  & 0.61 &  28\\
\hline\Tstrut

21. {\tt Gould\_Scott/D2Q06C}  &  5,831 &  2,171 &    33,081 & 0.63 & 34 \\
22. {\tt JGD\_Kocay/Trec14}    & 15,904 &  3,159 & 2,872,265 & 0.80 & 2,500\\
23. {\tt Meszaros/scsd8-2b}        & 35,910 &  5,130 & 112,770 & 0.86 &   514\\
24. {\tt LPnetlib/lp\_osa\_07}    &  25,067 & 1,118  &   144,812 & 0.96 & 6 \\
\hline
 \end{tabular}
\end{center}
}
\end{table}

\subsection{ILUP($p, \tau$) results}

\begin{table}[htbp]
\caption{Results for CGLS preconditioned by ILU($10,\tau$) with $\widetilde S$ given by (\ref{eq: S}) held and factorized as a dense matrix. 
$its$ is the number of CGLS iterations.
$psize$ denotes the storage for the matrices used
in applying the preconditioner (Algorithm~\ref{alg:apply_ilu}).
$ratio_{PT}$ is the estimate of the backward error given by \eqref{eq:ratio_PT}.
$nmod$ is the number of pivots that are modified. $\dag$ indicates requested accuracy not obtained.}
\label{T:results_dense}\vspace{3mm}
{
\footnotesize
\begin{center}
\tabcolsep=1.8mm
\begin{tabular}{l|rclr|rclr} \hline
Problem     &  \multicolumn{4}{c|}{$\tau=0.0$} &
\multicolumn{4}{c}{$\tau=0.1$} \\
& $its$ & $psize$ & $ratio_{PT}$ & $nmod$ 
& $its$ & $psize$ & $ratio_{PT}$  & $nmod$\\
\hline\Tstrut
1. 
& 2 & 1.65$\times 10^5$&1.84$\times 10^{-12}$& 163& 2 & 1.34$\times 10^5$&1.44$\times 10^{-11}$& 136\\

2. 
& 3 & 1.13$\times 10^4$&3.56$\times 10^{-11}$& 30 & 2 &8.49$\times 10^3$&1.02$\times 10^{-13}$& 44  \\

3. 
&  3 &3.31$\times 10^4$&2.93$\times 10^{-18}$& 162& 3 &3.24$\times 10^4$&1.83$\times 10^{-11}$& 173 \\

4. 
& 3 &4.12$\times 10^4$&5.13$\times 10^{-16}$&0&  2 & 1.82$\times 10^4$&8.70$\times 10^{-11}$&0\\

5. 
& $\dag$ &8.77$\times 10^5$&6.06$\times 10^{-7}$&0&  2 &9.24$\times 10^5$&7.33$\times 10^{-17}$&0\\

\hline\Tstrut
6. 
& 4 & 1.23$\times 10^6$&1.03$\times 10^{-13}$&0& 4 & 1.26$\times 10^6$&7.69$\times 10^{-12}$&0\\

7. 
& 2 &4.06$\times 10^7$&2.10$\times 10^{-17}$&0& $\dag$ &4.11$\times 10^7$&2.80$\times 10^{-6}$&0\\

8. 
&  3 & 1.03$\times 10^6$&3.36$\times 10^{-15}$&9& 2 &9.97$\times 10^5$&6.59$\times 10^{-12}$&6 \\

9. 
& 10 &5.92$\times 10^5$&1.83$\times 10^{-11}$&0&  10 &5.90$\times 10^5$&1.83$\times 10^{-11}$&0 \\ 

10. 
& 36 &5.89$\times 10^6$&8.63$\times 10^{-11}$&0&  5 &5.53$\times 10^6$&6.57$\times 10^{-12}$&205\\

\hline\Tstrut

11. 
& 10 &8.67$\times 10^6$&2.46$\times 10^{-11}$&0&  10 &8.67$\times 10^6$&2.46$\times 10^{-11}$&0\\

12. 
& 279 &2.22$\times 10^7$&2.98$\times 10^{-11}$&0&  146 &2.22$\times 10^7$&3.27$\times 10^{-11}$&0 \\

13. 
& $\dag$ &2.00$\times 10^7$&2.50$\times 10^{-5}$&0&  $\dag$ & 1.98$\times 10^7$&1.07$\times 10^{-5}$&0\\

14. 
& 121 &5.19$\times 10^6$&9.83$\times 10^{-11}$&0&  4 &4.63$\times 10^6$&1.72$\times 10^{-12}$&0 \\

15. 
& 21 &3.45$\times 10^6$&6.18$\times 10^{-11}$&0&  38 &2.97$\times 10^6$&2.49$\times 10^{-11}$&0 \\
\hline\Tstrut

16. 
& 10 &8.89$\times 10^7$&2.89$\times 10^{-11}$&0&  10 &8.89$\times 10^7$&2.62$\times 10^{-11}$&0 \\

17. 
& 8 & 1.24$\times 10^8$&4.98$\times 10^{-11}$&0&  2 & 1.23$\times 10^8$&9.69$\times 10^{-11}$&91 \\

18. 
&  6 & 1.21$\times 10^6$&4.95$\times 10^{-11}$&0& 46 & 1.14$\times 10^6$&8.46$\times 10^{-11}$&0\\

19. 
& $\dag$ &6.71$\times 10^5$ &1.42$\times 10^{-8}$& 0& $\dag$ & 6.64$\times 10^5$ &4.64$\times 10^{-10}$&0\\

20. 
& 29 & 1.55$\times 10^8$&3.38$\times 10^{-11}$&0& 3 & 1.55$\times 10^8$&5.38$\times 10^{-12}$&78 \\

\hline\Tstrut

21. 
&  43 &6.76$\times 10^6$&9.55$\times 10^{-11}$&0& 75 &6.76$\times 10^6$&9.59$\times 10^{-11}$&0\\

22. 
& 3 &8.14$\times 10^7$&1.38$\times 10^{-12}$&0 & 3 & 8.14$\times 10^7$ &2.72$\times 10^{-13}$& 0  \\

23. 
&  21 &4.74$\times 10^8$&9.20$\times 10^{-11}$&0& 19 &4.74$\times 10^8$&3.95$\times 10^{-11}$&0\\

24. 
& 19 &2.87$\times 10^8$ &7.08$\times 10^{-11}$&0& 19 &2.87$\times 10^8$&5.93$\times 10^{-11}$&0 \\
\hline
\end{tabular}
\end{center}
}
\end{table}

\begin{table}[htbp]
\caption{Results for CGLS preconditioned by ILU($10,\tau$) with
the solve in Step 2 of Algorithm~\ref{alg:apply_ilu} approximated by
(a) two iterations of the conjugate gradient method and
(b) $\widehat S=I$. The corresponding CGLS iteration counts are
$its_a$ and $its_b$ and 
$ratio_{PTa}$ and $ratio_{PTb}$ are the estimates of the backward errors \eqref{eq:ratio_PT}.
$psize$ denotes the storage $nz(\widetilde L_1)+
nz(\widetilde L_2) + nz(\widetilde U_1)$ for the matrices used
in applying the preconditioner (Algorithm~\ref{alg:apply_ilu}). 
$\dag$ indicates requested accuracy not obtained.
}
\label{T:results_CG}\vspace{3mm}
{
\footnotesize
\begin{center}
\tabcolsep=.9mm
\begin{tabular}{l|cccll|cccll} \hline
Problem     &  \multicolumn{5}{c|}{$\tau=0.0$} &
\multicolumn{5}{c}{$\tau=0.1$} \\
& $its_a$ & $its_b$ & $psize$ & $ratio_{PTa}$ & $ratio_{PTb}$   
& $its_a$ & $its_b$ & $psize$ & $ratio_{PTa}$ & $ratio_{PTb}$   \\
\hline\Tstrut
1. 
& 2 & 2 & 1.62$\times 10^5$ &7.51$\times 10^{-12}$&1.83$\times 10^{-11}$& 2 & 2 & 1.31$\times 10^5$ &9.13$\times 10^{-12}$&2.86$\times 10^{-13}$  \\

2. 
& 2 & 3 & 1.13$\times 10^4$ &2.36$\times 10^{-11}$&1.29$\times 10^{-11}$& 2 & 3 & 8.44$\times 10^3$   &4.88$\times 10^{-12}$&1.69$\times 10^{-14}$  \\

3. %
& 2 & 2 & 5.08$\times 10^4$ &7.94$\times 10^{-18}$&2.32$\times 10^{-17}$& 3 & 2 & 3.08$\times 10^4$  &1.13$\times 10^{-15}$&2.56$\times 10^{-16}$  \\

4. 
& 3 & 4 & 3.73$\times 10^4$ &2.02$\times 10^{-14}$&8.55$\times 10^{-13}$& 2 & 2 & 1.43$\times 10^4$ &7.23$\times 10^{-11}$&1.38$\times 10^{-12}$  \\

5. 
& 2 &$\dag$ & 7.44$\times 10^5$ &1.52$\times 10^{-16}$&6.17$\times 10^{-7}$& $\dag$ &$\dag$ & 7.91$\times 10^5$ &4.53$\times 10^{-6}$&9.38$\times 10^{-6}$  \\
\hline\Tstrut

6.
& 5    & 3 & 1.56$\times 10^5$ &9.25$\times 10^{-11}$&7.99$\times 10^{-13}$&  7 & 3 & 1.87$\times 10^5$ &9.50$\times 10^{-11}$&4.17$\times 10^{-16}$  \\

7. 
& $\dag$ & 2 & 4.30$\times 10^7$ &3.19$\times 10^{-5}$&8.92$\times 10^{-18}$& $\dag$ &2 & 3.36$\times 10^7$ &7.33$\times 10^{-7}$&1.85$\times 10^{-17}$  \\

8. 
& 3 & 2 & 3.71$\times 10^5$ &1.32$\times 10^{-16}$& 6.00$\times 10^{-13}$& 3 & 2 & 2.77$\times 10^5$  &5.08$\times 10^{-12}$&5.58$\times 10^{-14}$  \\

9. 
& 6 & 3 & 5.90$\times 10^4$ & 1.94$\times 10^{-12}$&4.79$\times 10^{-17}$& 6 & 3 & 5.73$\times 10^4$  &1.94$\times 10^{-12}$&4.79$\times 10^{-17}$  \\

10. 
& 3 &4 & 4.75$\times 10^5$ &1.49$\times 10^{-11}$&9.82$\times 10^{-11}$& 4 &35 & 1.21$\times 10^5$ &6.44$\times 10^{-12}$&1.34$\times 10^{-12}$  \\
\hline\Tstrut

11. 
& 6 &4 & 2.45$\times 10^5$ &3.21$\times 10^{-12}$&2.33$\times 10^{-19}$& 6 &11 & 2.46$\times 10^5$ &3.21$\times 10^{-12}$&7.52$\times 10^{-11}$  \\

12. 
& 23 &4 & 1.43$\times 10^5$ &9.24$\times 10^{-11}$&8.98$\times 10^{-12}$& 28 &6 & 1.43$\times 10^5$  &2.72$\times 10^{-11}$&6.24$\times 10^{-12}$  \\

13. 
& 7 &2 & 3.52$\times 10^5$ &2.89$\times 10^{-11}$&9.65$\times 10^{-12}$& 6 &3 & 1.74$\times 10^5$ &4.05$\times 10^{-11}$&3.49$\times 10^{-13}$  \\

14. 
& 4 &3 & 6.51$\times 10^5$ &6.61$\times 10^{-14}$&1.56$\times 10^{-16}$& 5 &5 & 8.28$\times 10^4$  &1.09$\times 10^{-11}$&6.10$\times 10^{-12}$  \\

15. 
& 7 &3 & 5.41$\times 10^5$ &4.72$\times 10^{-11}$&8.92$\times 10^{-17}$& 5 &3 & 5.74$\times 10^4$  &2.62$\times 10^{-12}$&8.82$\times 10^{-13}$  \\
\hline\Tstrut

16. 
& 3 &3 & 9.81$\times 10^4$ &1.57$\times 10^{-12}$& 2.51$\times 10^{-11}$& 3 &3 & 9.14$\times 10^4$  &1.79$\times 10^{-12}$&2.51$\times 10^{-11}$  \\

17. 
& 6 &5 & 3.56$\times 10^5$ &1.94$\times 10^{-11}$&1.27$\times 10^{-12}$& 3 &3 & 1.54$\times 10^5$  &7.29$\times 10^{-12}$&6.50$\times 10^{-13}$  \\

18. 
& 4 &3 & 1.64$\times 10^5$ &4.02$\times 10^{-14}$&1.08$\times 10^{-12}$& 5 &3 & 5.49$\times 10^4$ &2.38$\times 10^{-11}$&3.94$\times 10^{-15}$  \\

19. 
& 5 &4 & 2.28$\times 10^4$ &3.53$\times 10^{-11}$&3.81$\times 10^{-11}$& 12 &10 & 1.63$\times 10^4$ &2.92$\times 10^{-11}$&2.39$\times 10^{-11}$  \\

20. 
& 7 &3 & 3.22$\times 10^5$ &5.51$\times 10^{-11}$&1.80$\times 10^{-14}$& 3 &$\dag$ & 1.61$\times 10^5$  &1.16$\times 10^{-12}$&1.73$\times 10^{-6}$  \\

\hline\Tstrut

21. 
& 4 &4 & 6.30$\times 10^4$ &5.31$\times 10^{-11}$&4.33$\times 10^{-14}$& 7 &3 & 5.89$\times 10^4$ &1.00$\times 10^{-11}$&6.73$\times 10^{-11}$  \\

22. 
&  4 &3 & 1.65$\times 10^5$ &3.42$\times 10^{-11}$&4.21$\times 10^{-16}$& 3 &3 & 1.55$\times 10^5$ &1.37$\times 10^{-14}$&3.66$\times 10^{-13}$  \\

23. 
&  5 &3 & 1.62$\times 10^5$ &6.93$\times 10^{-12}$&4.59$\times 10^{-11}$& 7 &4 & 1.61$\times 10^5$ &4.05$\times 10^{-12}$& 1.72$\times 10^{-11}$  \\

24. 
& 3 & 4 & 1.35$\times 10^4$  &6.57$\times 10^{-11}$&3.01$\times 10^{-13}$&3 & 4 & 1.35$\times 10^4$ &5.71$\times 10^{-11}$&3.00$\times 10^{-13}$   \\
\hline
\end{tabular}
\end{center}
}
\end{table}

Tables~\ref{T:results_dense} 
and \ref{T:results_CG} present results for ILUP($p, \tau$) 
with $p = 10$ and $\tau = 0$ and 0.1. The parameter $small$ is set to $10^{-10}$ and the pivoting threshold is $\mu = 0.1$.
In Table~\ref{T:results_dense}, the results are for explicitly computing and factorizing the dense matrix $\widetilde S$; those in  Table~\ref{T:results_CG} are for solving the system $\widetilde S w = u$
at Step 2 of Algorithm~\ref{alg:apply_ilu} using
(a) two iterations of the conjugate gradient method and
(b) replacing $\widetilde S$ by the identity.
For each test, we report the number of CGLS iterations and the storage $psize$ needed for the matrices used in applying the preconditioner
(Algorithm~\ref{alg:apply_ilu}). It is clear that if $\widetilde S$ is held as a dense matrix then as $(m-n)/m$ increases, the storage for its factors 
dominates and makes the method impractical for large problems.
Avoiding forming $\widetilde S$ (Table~\ref{T:results_CG}) 
has the potential to  allow much larger problems to
be solved, with the memory determined primarily by the choice of $p$.

 While larger values of $p$ increase the storage requirements, for most of our test matrices, this did not give a reduction in the CGLS iteration count and so detailed results for other values are not reported.
The counts  and storage can be sensitive to the choice of the dropping threshold $\tau$
but not always in a consistent or predictable way. 
For some examples, including problems 3 and 4,  $\tau = 0.1$ leads to significantly sparser $\widetilde L$ and $\widetilde U$, without adversely affecting the
preconditioner quality. But increasing $\tau$ does not necessarily
give sparser factors and smaller $psize$ (for instance,
problem 6). This is because dropping more entries from the current
column $j$ may not result in the remaining columns having fewer entries of small absolute value. For some problems,  $\tau>0$ can reduce the
iteration count, even if more pivots are modified during the
ILU factorization (this is seen in Table~\ref{T:results_dense} for
problem 10). We note that a significant number
of pivot modifications during the ILU factorization can be made and still give a preconditioner
that results in rapid convergence  of CGLS, illustrating the robustness 
and effectiveness of the proposed ILU algorithm.

An unexpected finding is that using a small number of CG iterations
to solve the system $\widetilde S w = u$ (or even replacing $\widetilde S$ by the identity) can perform significantly better than factorizing $\widetilde S$. This is seen even for problems
that are not quasi square.
We observe that similar effects have been reported when employing some algebraic preconditioners for solving linear systems of equations, in particular, nonsymmetric incomplete factorization or approximate inverse preconditioners. Allowing more fill-in can sometimes be counter-productive; see a comparison 
of the level-based preconditioners ILU(0) and ILU(1) reported in \cite{beht:00}, which  motivated later research into possible benefits of pre-sparsifying the system matrix \cite{sctu:2011}.

While the ILU preconditioner performs well
in many of our tests and $ratio_{PT}$ achieves the requested accuracy after a
small number of iterations, this is not always the case. In 
the following figures, we plot $ratio_{PT}$ against the CGLS iteration count (either until convergence or 2000 iterations, whichever occurs first).
Figure~\ref{fig:large_estimate} illustrates the case that 
$ratio_{PT}$ decreases steadily whereas
Figure~\ref{fig:GEestimate} shows that for problems 5 and 13,
after the initial iterations, $ratio_{PT}$ decreases very slowly
and becomes close to stagnating;
this is most likely because the maximal attainable accuracy has been (approximately) reached \cite{gree:97a,pati:2024}.
For a small number of our test cases, $ratio_{PT}$ initially does not decrease smoothly but then rapidly decreases to the required accuracy. This is illustrated in Figure~\ref{fig:psse1estimate}.
We are not able to predict this behaviour a priori.

\begin{figure}[htbp]
\begin{center}
        \includegraphics[scale=0.45]{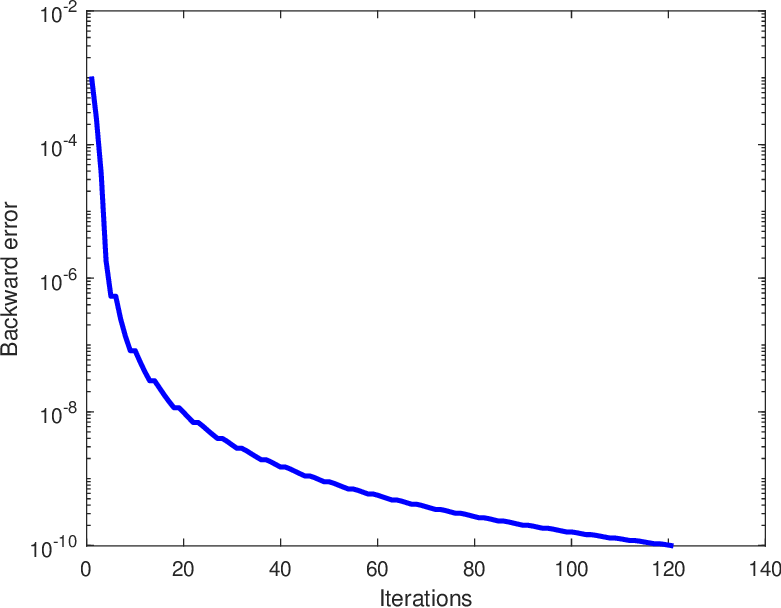}
\end{center}
\caption{The convergence of the backward error estimate $ratio_{PT}$ for
the problem 14  with preconditioner ILU(10,0.0). The matrix $\widetilde S$ is constructed and factorized as a dense matrix.}
\label{fig:large_estimate}
\end{figure}

\begin{figure}[htbp]
\begin{center}
        \includegraphics[scale=0.45]{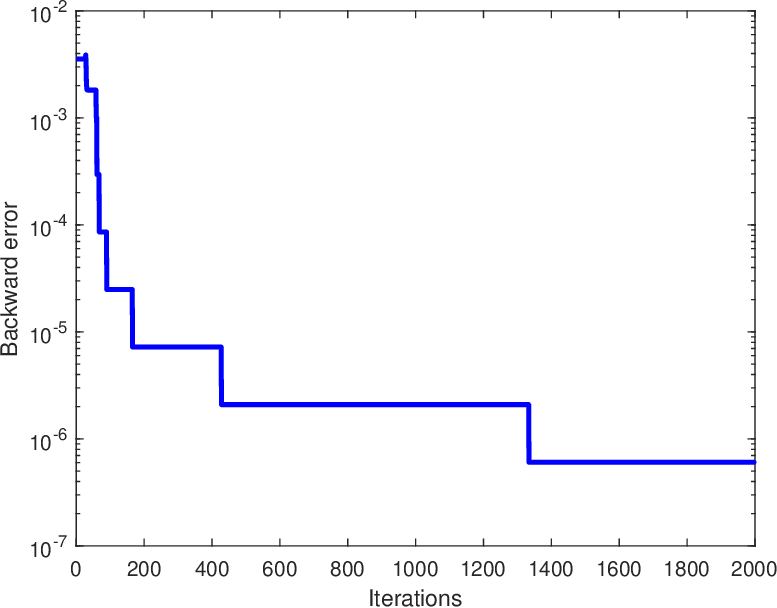}
 \hspace{1cm}
        \includegraphics[scale=0.45]{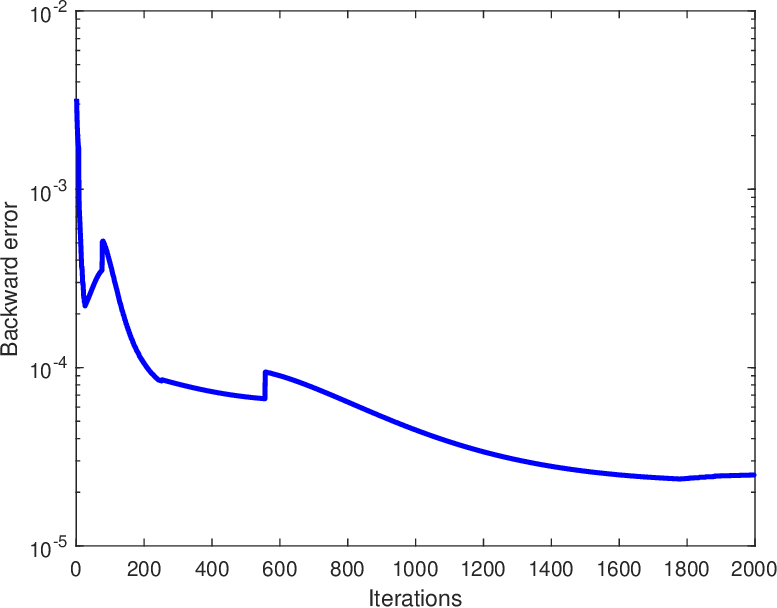}
\end{center}
\caption{The convergence of the backward error estimate $ratio_{PT}$ for
problems 5 and 13  with preconditioner ILU(10,0.0). The matrix $\widetilde S$ is constructed and factorized as a dense matrix.}
\label{fig:GEestimate}
\end{figure}

\begin{figure}[htbp]
\begin{center}
    \includegraphics[scale=0.45]{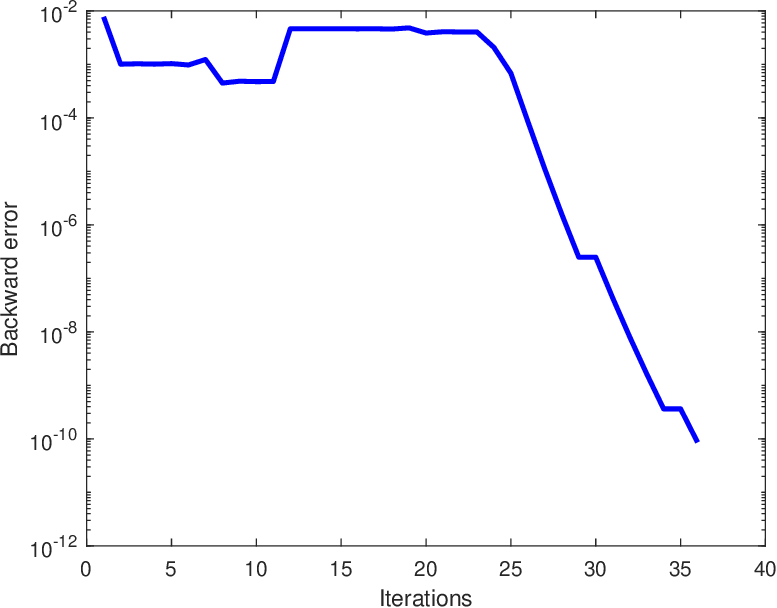} \hspace{1cm}
        \includegraphics[scale=0.45]{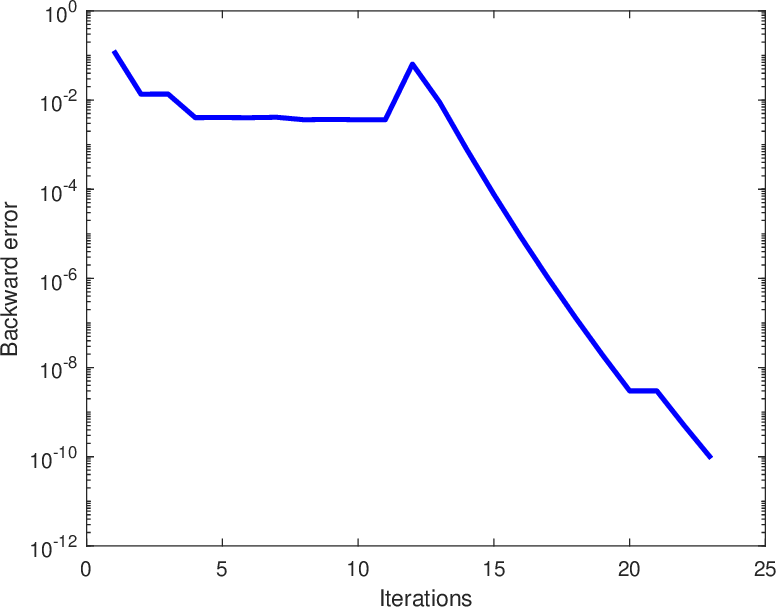}

\end{center}
\caption{The convergence of the backward error estimate $ratio_{PT}$ for
problems 10 and 12 with preconditioner ILU(10,0.0). For problem 10 (left) the matrix $\widetilde S$ is constructed and factorized as a dense matrix; for problem 12 (right),
Step 2 of Algorithm~\ref{alg:apply_ilu} is approximated by
two iterations of the conjugate gradient method. }
\label{fig:psse1estimate}
\end{figure}

\subsection{Comparison with an IC preconditioner}

It is beyond the scope of this paper to compare the proposed ILU preconditioner with a range of other preconditioners for LLS problems that have been proposed in the literature.
However, it is of interest to include some results for an incomplete Cholesky (IC) factorization preconditioner that was the most effective of the least squares preconditioners reported on in \cite{gosc:2017}. This memory-limited IC-based preconditioner is based on computing
$C+ \alpha I \approx \widetilde L_{IC} \widetilde L_{IC}^T$, where
$C = A^TA$  and $\alpha>0$ is chosen to prevent breakdown of the factorization. The user controls the amount of memory used in the construction of
$\widetilde L$ as well as the maximum number of entries in each column of $\widetilde L_{IC}$.
Details of the approach are given in \cite{sctu:2014a,sctu:2014b}.
We use the implementation available as the HSL package {\tt HSL\_MI35}.

A number of our test problems contain one or more rows that have many more entries
than the other rows (see Table~\ref{T:test problems}).
The experimental results in the last section demonstrate that the ILU preconditioner
is able to handle these as well as it handles fully sparse problems. Importantly, it is unnecessary to 
explicitly determine which rows should be classified as dense: the proposed ILU approach is agnostic to the
sparsity pattern of $A$ and requires no preordering of $A$. This is not the case for preconditioners that are based on 
the normal matrix, for which it is necessary to handle the dense rows separately (recall the discussion in Section~\ref{sec:sparse-dense}); thus we do
not report results for {\tt HSL\_MI35} for these problems.

 \begin{table}[htbp]
\caption{Results for CGLS with the IC preconditioner {\tt HSL\_MI35}. 
$its$ is the number of CGLS iterations.
$nz(\widetilde L_{IC})$ denotes the storage for incomplete Cholesky factor.
$ratio_{PT}$ is the estimate of the backward error given by \eqref{eq:ratio_PT}.
$\dag$ indicates requested accuracy not obtained.}
\label{T:results_IC}\vspace{3mm}
{
\footnotesize
\begin{center}
\tabcolsep=7mm
\begin{tabular}{l|rclr} \hline
Problem  & $its$ & $nz(\widetilde L_{IC})$ & $ratio_{PT}$ \\
\hline\Tstrut
5.   & $\dag$   &  8.15$\times 10^4$ &    7.765$\times 10^{-5}$  \\
\hline\Tstrut

7.   & $\dag$   &  4.78$\times 10^5$ &    1.426$\times 10^{-4}$  \\ 
10.  &        6 &  3.38$\times 10^5$ &    9.073$\times 10^{-11}$  \\
\hline\Tstrut

12.  &       12 &  4.55$\times 10^5$ &    4.880$\times 10^{-11}$  \\ 
13.  &        9 &  3.12$\times 10^5$ &    6.729$\times 10^{-11}$ \\ 
14.  &        7 &  1.28$\times 10^5$ &    9.730$\times 10^{-12}$  \\ 
15.  &        4 &  9.59$\times 10^4$ &    4.971$\times 10^{-14}$  \\
\hline\Tstrut

17.  &       5  &  3.29$\times 10^5$ &    1.428$\times 10^{-11}$  \\ 
18.  &       16 &  2.96$\times 10^4$ &    9.609$\times 10^{-11}$  \\ 
19.  &        3 &  1.14$\times 10^4$ &    9.494$\times 10^{-14}$  \\
20.  &       12 &  3.36$\times 10^5$ &    9.431$\times 10^{-11}$  \\
\hline\Tstrut

21.  &       43 &  6.67$\times 10^4$ &    9.874$\times 10^{-11}$  \\ 
24.  &       92 &  3.42$\times 10^4$ &    2.442$\times 10^{-11}$  \\ 
\hline
\end{tabular}
\end{center}
}
\end{table}
In Table~\ref{T:results_IC}, we present results for CGLS left preconditioned using {\tt HSL\_MI35}. The parameters $\tt lsize$ and $rsize$ that control the memory requirements and size of $\widetilde L_{IC}$ are set to 30 and the drop tolerances are set to zero.
Comparing the results to those in Table~\ref{T:results_CG}, we see that neither approach
consistently outperforms the other, although for a number of examples (including problems 23 and 24) the ILU($10,\tau$) preconditioner
has a lower iteration count, while using less memory. We conclude that
the new ILU preconditioner potentially offers an  attractive alternative to the IC preconditioner.

\section{Concluding remarks}
\label{sec: conclusions}
In this work, we have introduced a class of ILU‑based row‑splitting preconditioners for large sparse linear least‑squares problems. Our approach uses an incomplete LU factorization to identify a well‑conditioned square submatrix 
$A_1$ of $A$ and to then construct an algebraic preconditioner that avoids forming or factorizing the normal matrix 
$A^TA$. This makes the method particularly attractive for problems in which 
the normal matrix is ill-conditioned and/or
 is significantly denser than $A$.

The proposed framework unifies several ideas: additive correction formulations, Woodbury‑type updates, and sparse ILU factorizations with threshold pivoting. By exploiting the structure of the row splitting, applying the preconditioner requires triangular solves with the sparse incomplete factors of 
$A_1$ and solving an auxiliary symmetric positive definite system involving the matrix 
$S$ given by (\refeq{eq:S}). When the least‑squares matrix $A$ is quasi‑square, $S$ is inexpensive to assemble and factorize. Instead, the auxiliary  system
can be solved approximately using a small number of steps of an iterative method or, as our test results show, it can be avoided altogether by using $S \approx I$.

Our numerical experiments demonstrate that the proposed ILUP preconditioners are competitive with, and sometimes superior to, incomplete Cholesky preconditioners applied to the normal equations. The method also performs well on sparse‑dense problems, where the row‑splitting naturally isolates the dense rows.

The flexibility of the approach allows for several extensions. The ILU factorization can be updated efficiently when rows are added or removed, making it suitable for dynamic or streaming least‑squares settings. Future work will explore the use of mixed precision arithmetic for reducing the cost of the factorization
and/or limiting the memory required for holding the incomplete factors. This will
build on our recent experience of constructing incomplete Cholesky factorization
preconditioners using mixed precision~\cite{sctu:2024,sctu:2025a}.

Overall, the results indicate that ILU‑based row‑splitting preconditioners offer a practical, scalable, and robust alternative for solving large sparse least‑squares problems, particularly in applications where forming the normal equations is either undesirable or infeasible.

\def\cprime{$'$} \def\cprime{$'$} \def\cprime{$'$}


\end{document}